\documentclass[a4paper,12pt]{amsart}

\pdfoutput=1
\usepackage{microtype}
\usepackage{amsmath,amssymb,amsthm,amsfonts}
\usepackage{mathtools}
\usepackage{hyperref}
\usepackage{enumitem}
\usepackage{parskip}

\usepackage{kpfonts}
\usepackage[T1]{fontenc}
\DeclareMicrotypeAlias{kpfonts}{pxfonts}

\hypersetup{pdfauthor={Abel B. Stern},pdftitle={Finite-rank approximations of
    spectral zeta residues}}

\newcommand{\Z}{\mathbb{Z}} 
\newcommand{\R}{\mathbb{R}} 
 
\newcommand{\br}[1]{\left(#1\right)}

\newcommand{\eqdef}{\overset{\underset{\mathrm{def}}{}}{=}}

\DeclareMathOperator{\re}{Re}
\renewcommand*{\Re}{\re}

\DeclareMathOperator{\res}{res}
\DeclareMathOperator{\supp}{supp}

\DeclareMathOperator{\tr}{tr}

\DeclareMathOperator{\vol}{vol}
\DeclareMathOperator{\End}{End}

{\theoremstyle{plain}
\newtheorem{thm}{Theorem}[section]

\newtheorem{prop}[thm]{Proposition}
\newtheorem{corr}[thm]{Corollary}}

{\theoremstyle{remark}

\newtheorem{remark}[thm]{Remark}}

{\theoremstyle{definition}
\newtheorem{dfn}[thm]{Definition}}

\newlength{\stdsummationwidth}
\author{Abel B. Stern}
\address{Institute for Mathematics, Astrophysics and Particle Physics \\ Radboud University \\ P.O. Box 9010 \\ 6500GL Nijmegen, the Netherlands}
\urladdr{https://orcid.org/0000-0003-0151-6394}
\thanks{The author was supported by FOM Vrij Programma No. 150.}

\title[Finite-rank approximations of spectral zeta residues]{Finite-rank approximations \\ of spectral zeta residues}

\begin{document}

\begin{abstract}
We use the asymptotic expansion of the heat trace to express all residues of spectral zeta functions as regularized sums over the
spectrum. The method extends to those spectral zeta functions that are localized by a bounded operator.
\end{abstract}

\maketitle

\section{Introduction}

The spectral theory of elliptic operators presents a major connection
between functional analysis and differential geometry. It provides a number of
interesting relations between the spectrum with multiplicities (which is the complete unitary invariant of self-adjoint operators) and the symbol (which is directly tied to the local expression of pseudodifferential operators). Thereby, it shows how the local structure of such an operator influences its global properties, and vice versa.  Of particular interest is the relation between the symbol of an elliptic operator and its spectral asymptotics that is conveyed by the spectral zeta residues.

This note is concerned with expressing the spectral zeta residues as a limit of
partial sums of particular functions over the spectrum. Together with the
well-known relations between the zeta residues, the symbol, and the heat
expansion, this bridges a gap between the continuum setup of spectral geometry
and the finite objects in combinatorial geometry and computer science. For
instance, this method allows one to approximate the scalar curvature on a
compact Riemannian manifold using only a finite part of its Laplacian spectrum.

We will first provide some background material in Section
\ref{subsec:background} and then introduce the main topic and results of this
note in Section \ref{subsec:summary}.

\subsection{Background: spectral zeta residues and Weyl's law} \label{subsec:background}
In 1949, Minakshisundaram and Pleijel showed \cite{Minakshisundaram1949} that the zeta function \[
\zeta(\Delta,s) \eqdef \tr \Delta^{-s}
\]
of the Laplace operator on a compact Riemannian $d$-dimensional manifold can be meromorphically extended to the complex plane, with simple poles occurring in the points $d/2-\Z_{\geq 0} \subset \R$. 
The residues at these poles (which are proportional to the so-called heat kernel coefficients) relate the spectrum of the Laplace operator, itself an isometry invariant, to other known isometry invariants of the manifold, such as its volume and scalar curvature. 

More generally, an elliptic pseudodifferential operator defines a spectral zeta function, by the work of Seeley \cite{Seeley1967}, and its residues relate geometric information to the operator spectrum in the following way.
If $\Delta$ is a positive elliptic classical pseudodifferential operator of order $m \in \R_{\geq 0}$ and $k$ is any nonnegative integer, the residue at $s=(d-k)/m$ of $\zeta(\Delta,s)$ equals the Wodzicki residue 
\[
    \frac{1}{m} \int_{S^* M} \tr a_{-n}^{(k-d)/m} (x,\xi) d^{n-1} \xi dx,
\]
where $a_{-n}(x,\xi)$ is the homogenous term of order $-n$ in the decomposition of the symbol of $\Delta$ and $S^* M \subset T^* M$ is the cosphere bundle, cf. \cite{Wodzicki1987,Gilkey1994,Connes}. In the inversely Mellin transformed picture, the residues correspond to integrals of terms in the asymptotic expansion of the integral kernel of $e^{-t \Delta}$ along the diagonal, as $t \to 0^+$.

One wonders which conclusions about the operator spectrum can be drawn from the
zeta residues, apart from the rather opaque one provided by their definition.
The most well-known result in this direction is provided by the Wiener-Ikehara
theorem, which relates the first residue of the zeta function to the asymptotics
of the number $N(\Lambda)$ of eigenvalues smaller than $\Lambda$. If $\alpha$ is
a monotone increasing function and the zeta function
\[
    \zeta(z) = \int_1^\infty \lambda^{-z} d \alpha(\lambda)
\]
converges for $\Re z > 1$ and can be meromorphically extended to $\Re z \geq 1$ with a simple pole at $z=1$, then the Wiener-Ikehara theorem states that $\alpha(\Lambda) \sim \Lambda \res_{z=1} \zeta(z) $ as $\Lambda \to \infty$. By Seeley's results on complex powers of elliptic operators, we may apply this theory to the spectral zeta function of a degree $m$ elliptic pseudodifferential operator on a $d$-dimensional manifold. As in \cite{Minakshisundaram1949}, this yields Weyl's famous law
\[
    N(\Lambda) = \res_{s=d/m} \zeta(\Delta,s) \Lambda^{d/m} + o(\Lambda^{d/m}).
\]
Note that this restriction only uses the first residue of the zeta function: the zeta series $\tr \Delta^{-s}$ is absolutely convergent for $\Re s > d/m$.

The problem of improving on the accuracy in Weyl's law has attracted much
attention over the last century. That is, one wonders whether we can obtain an
asymptotic expansion of $N(\Lambda) - \res_{s=d/m} \zeta(\Delta,s) \Lambda^{d/m}
$ for specific classes of operators. However, sharp bounds that depend only on the spectral zeta function of $\Delta$ have not yet been produced, even in the well-studied case of the
Laplacian on flat tori. It would seem natural that the Wiener-Ikehara result could be extended to relate the asymptotic expansion of $N(\Lambda)$ to the location of the poles of the zeta function. However, this approach is limited by the difficulties of inverse Mellin and Laplace transforms, see e.g. \cite{Aramaki1996}, \cite[9.7.2]{Berger2003}. The lower poles of the zeta function can therefore not yet be related to the asymptotics of $N(\Lambda)$. However, we will explain in the present paper how to relate their residues to the asymptotics of other functionals of the operator spectrum.

\subsection{Zeta residues as a resummation of the spectrum} \label{subsec:summary}
In the converse direction to Weyl's law, we ask which conclusions one can draw about the residues, given access to increasingly large \emph{finite} subsets of the operator spectrum. For the first residue, Weyl's law gives rise to the Dixmier trace\footnote{As restricted to the ideal on which this formula converges. For a treatment of the Dixmier trace on the larger ideal $L^{(1,\infty)}$, see \cite{Carey2007}.} formula: if $\lambda_0 \leq \lambda_1 \leq \dotsb$ are the eigenvalues of $\Delta$, then
\[
    \res_{s=d/m} \zeta(\Delta,s) = \lim_{N \to \infty} \frac{ \sum_{i=1}^N \lambda_i^{-d/m} }{\log N} = \lim_{\Lambda \to \infty} \frac{ \sum_{\lambda < \Lambda} \lambda^{-d/m} }{\frac{d}m \log \Lambda}.
\]
The Dixmier trace is singular: it vanishes if $\tr \Delta < \infty$. However, if $\Delta_j$ is a sequence of finite rank operators of operator norm $\| \Delta_j \| \eqdef \Lambda_j \to \infty$ such that
\[
    \tr \Delta_j^{-s_0} - \sum_{0 <  \lambda < \Lambda_j} \lambda^{-s_0} = o( \log \Lambda_j),
\]
the Dixmier trace of $\Delta^{-s_0}$ clearly equals $\lim_{j \to \infty} \tr { \Delta_j^{-s_0} } / \log \Lambda_j$. It would seem plausible that this result can be extended to further poles, i.e. that there are also functionals $c_k$ on the finite rank operators such that 
\[
    \lim_{j \to \infty} c_k(\Delta_j) = \res_{s = s_k} \zeta(\Delta,s),
\]
if $\zeta(\Delta,s)$ has poles at $s_0 > \dotsb > s_k $ and the finite-rank $\Delta_j$ converge to $\Delta$ in some appropriate sense.
The main question, then, is whether we can write these functionals down explicitly. 

The existence of such asymptotic residue functionals is shown by our Theorem
\ref{thm:maintheorem}, and explicit expressions follow from the conditions
of Propositions \ref{prop:convolutionasymptotics} and \ref{prop:partialtrace}.
Corollary \ref{corr:firstpole} uses the theorem to improve on the convergence of the Dixmier trace formula, and Corollary \ref{corr:secondpole} exhibits the resulting expression for the second pole. Finally, Theorem \ref{thm:localization} is a simple extension of our treatment to the localized residue traces $\res_{s=s_k} \tr h \Delta^{-s}$, for any bounded operator $h$.

\section{Zeta residues as normal functionals}
We can ask how spectral zeta residues relate to finite subsets of the operator spectrum without referring to the original setting of differential geometry. We will henceforth consider the question in such generality, but we will need to restrict ourselves to operators whose zeta functions share an essential property with the Minakshisundaram-Pleijel zeta function.
\begin{dfn} \label{dfn:spectrallyelliptic}
    A positive, invertible\footnote{If not, just restrict to the complement of
      the kernel.}, self-adjoint, unbounded operator $\Delta$ with compact
    resolvent is said to be \emph{spectrally elliptic} if the `heat trace' $\tr
    e^{- t \Delta}$ admits an asymptotic expansion $\sum_{i=0}^\infty c_i t^{-
      s_i}$ as $t \to 0^+$, where the $s_i$ are decreasingly ordered reals. We
    say that $\Pi = \{s_i\}$ is the \emph{heat spectrum} of $\Delta$.
\end{dfn}

The definition implies that there is some $s_0 \in \R$ such that $\Gamma(s) \tr
\Delta^{-s}$ converges for $\Re s > s_0$ and can be analytically continued to a
meromorphic function $\Gamma(s) \zeta(\Delta,s)$ whose poles are all simple and
located in $\Pi \subset (-\infty,s_0]$.

This particular terminology was chosen because the motivating example of such operators are the positive elliptic differential\footnote{The classical elliptic
  \emph{pseudo}differential operators may have logarithmic terms in the heat
  expansion, leading to double poles of $\Gamma(z) \zeta(z)$ (e.g. at negative
  integers), and are thus excluded in general. However, see remark
  \ref{rmk:logarithmicterms}.} operators, cf. \cite{Seeley1967,Duistermaat1975}.
Indeed, for a positive elliptic differential operator of order $m$ on a
$n$-dimensional manifold, the heat spectrum is contained in
$\frac{n}{m},\frac{n-1}{m},\dotsc,\frac1{m}$.

Definition \ref{dfn:spectrallyelliptic} suggests that we will in fact
be technically concerned with the asymptotic heat trace coefficients $c_i$, and
the reader who is more familiar with that terminology may rest assured that they
are indeed what we are talking about. However, because
zeta residues are mathematically the more general notion, the text will
mainly refer to the coefficients by that name.

Our main theorem shows how the zeta residues of a spectrally elliptic operator are related to restrictions of its spectrum. We will provide the proof in Section \ref{sec:asymptoticseries}, below.

The precise asymptotic formula for the residue $\res_{s=s_k} \zeta(\Delta,s)$ at
some pole $s_k \in \Pi$ depends only on the location of the `previous' poles,
that is, on the set $\Pi \cap [s_k,\infty)$. The Dixmier trace, for instance,
requires knowledge of $s_0$ and of the fact that no poles $s_{-1} > s_0$ exist.
For a given finite set $\{s_i\}$ of such poles, all operators with zeta residues contained in this set can be considered simultaneously. Hence the following definition.

\begin{dfn}
For any finite set $\{s_i\}_{i=0}^k$ of decreasingly ordered reals, let $\mathcal{D}(\{s_i\}_{i=0}^k)$ be
the set of all spectrally elliptic operators $\Delta$ whose heat spectrum $\Pi$
satisfies $\Pi \cap [s_k,\infty) \subset \{s_i\}$, i.e., whose heat trace admits an asymptotic expansion $\sum_{i=0}^k c_i t^{-s_i} +
O(t^{-s_{k+1}})$ as $t \to 0^+$, for some $s_{k+1} < s_k$.
\end{dfn}

If $\sigma(\Delta)$ is the (necessarily point) spectrum of some spectrally elliptic operator $\Delta$ and $F$ is any Borel measurable function, we will denote the summation of $F$ over the strictly positive eigenvalues smaller than $\Lambda$ by
\[
 \tr_\Lambda F(\Delta) \eqdef \sum_{\lambda \in \sigma(\Delta) \cap (0,\Lambda] } F(\lambda).
\]

\begin{thm}\label{thm:maintheorem}
    For any finite set $\{s_i\}_{i=0}^k$ of decreasingly ordered reals there exists a
    function $F$ such that, for all $\Delta \in \mathcal{D}(\{s_i\}_{i=0}^k)$,
    \begin{equation} \label{eq:mainformula}
      \res_{s=s_k} \zeta(\Delta,s) \Gamma(s) = \lim_{\Lambda \to \infty}  \epsilon(\Lambda)^{s_k} \tr_\Lambda F(\Delta \epsilon(\Lambda)),
    \end{equation}
    where $\epsilon(\Lambda) \eqdef m \log \Lambda / \Lambda$ for any $m > s_0 - s_k$.
\end{thm}

The problem posed in the introduction is then trivially resolved by the following
corollary, together with an explicit choice of $F$.

\begin{corr}\label{corr:mainquestion}
  Let $\{\Delta_j\}$ be any sequence of finite rank positive operators of norm $\| \Delta_j \| \eqdef \Lambda_j$ such that 
  \[
    \tr_\Lambda F(\Delta \epsilon(\Lambda))- \tr F(\Delta_j \epsilon(\Lambda_j)) = o((\epsilon(\Lambda_j)^{-s_k}).
  \]
  Then,
  \[
    \res_{s=s_k} \zeta(\Delta,s) \Gamma(s) = \lim_{j \to \infty} \epsilon(\Lambda_j)^{s_k}  \tr F( \Delta_j \epsilon(\Lambda_j)).
  \]
\end{corr}

\begin{remark}
  A uniform bound over $\mathcal{D}(\{s_i\}_{i=0}^k)$ on the rate of convergence is too
  much to ask: it depends on the complete spectrum of $\Delta$. However, the
  convergence rate will be shown to always be $O(\Lambda^{-1})$ if the next pole
  is at $s_k - 1$.
  For examples of arbitrary error for given $\Lambda$, see for instance \cite{ColindeVerdiere1987,Lohkamp1996}.
\end{remark}

\begin{remark} \label{rmk:logarithmicterms}
  The exclusion of logarithmic terms in the heat expansion from Definition
  \ref{dfn:spectrallyelliptic} corresponds to an exclusion of higher order
  poles of $\Gamma(s) \zeta(\Delta,s)$. However, Theorem \ref{thm:maintheorem}
  is unchanged if higher order poles are allowed, as long as they lie below
  $s_k$. Moreover, the approach can probably be modified to accomodate such
  logarithmic terms, at the expense of brevity and asymptotic rate of convergence.
\end{remark}

\subsection{Asymptotic series for zeta function residues} \label{sec:asymptoticseries}
For any spectrally elliptic operator $\Delta$, the function $\zeta(\Delta,s) \Gamma(s)$ is the Mellin transform of $e^{- t \Delta}$. Therefore, the asymptotic expansion of $e^{- t \Delta}$,
\[
  \tr  e^{-t \Delta} \sim \sum_{i=0}^k c_i t^{- s_i}   + o(t^{-s_k}),
\]
has coefficients
\[
  c_i = \res_{s=s_i} \zeta(\Delta,s) \Gamma(s).
\]

The proof of Theorem \ref{thm:maintheorem} relies on the asymptotic expansion of a Mellin convolution integral with $\tr  e^{- t \Delta}$. 
This strategy was inspired by the spectral action principle of \cite[Thm 1.145]{Connes2008a} and the approach of \cite{Aramaki1996}. However, here we have specifically constructed the function we convolve with such that its first moments vanish, which allows us to recover the zeta residue in the leading term.

\begin{proof}[Proof of Theorem \ref{thm:maintheorem}] \label{pf:maintheorem}
  The proof proceeds in two elementary steps. First, we show (Proposition \ref{prop:convolutionasymptotics}, below) that for suitable functions $f$ one has
  \[
    \int_0^\infty \tr  e^{-\epsilon t \Delta} f(t) = c_k \epsilon^{-s_k} + O(\epsilon^{-s_{k+1}})
  \]
  as $\epsilon \to 0$. This part requires the asymptotic expansion of $\tr   e^{-t \Delta}$, implying in particular that the poles of $\Gamma(s) \zeta(\Delta,s)$ must be simple and discrete (i.e. $\Delta$ has simple dimension spectrum disjoint with the negative integers) and lie on the real line. 

  Then, we use the Laplace transform $F$ to rewrite the integral as a trace,
  \[
    \int_0^\infty \tr  e^{- \epsilon t \Delta} f(t) dt = \tr  F(\epsilon \Delta),
  \]
  and we prove in Proposition \ref{prop:partialtrace} that there is a choice
  $\epsilon(\Lambda) = \Lambda^{-1} \log \Lambda^m $ such that $F$ decays as
  \[
    (\tr - \tr_\Lambda)  F(\Delta \epsilon(\Lambda)) = O(\epsilon^{-s_{k+1}}),
  \]
  so that we obtain the useful asymptotic behaviour
  \[
    \tr_\Lambda F(\Delta \epsilon(\Lambda)) =  c_k \epsilon(\Lambda)^{-s_k}  + O(\epsilon^{-s_{k+1}}).
  \]
  This second part relies on the finite summability of $\Delta$, that is, on the
  fact that $\tr e^{- t \Delta} = O(t^{-s_0})$ as $ t \to 0^+$.
\end{proof}

The following asymptotic estimate is completely straightforward.
\begin{prop} \label{prop:convolutionasymptotics}
  Let $f$ be a piecewise continuous function supported in $(1,\infty)$ that is
  $O(t^{-m})$ for all $m \in \R$ towards $\infty$. Let $\Delta \in \mathcal{D}(\{s_i\}_{i=0}^k)$.
  Then, the Mellin convolution integral of $f$ with $\tr e^{-t \Delta}$ has an
  asymptotic expansion
  \[
    \int_0^\infty \tr e^{- \epsilon t \Delta} f(t) dt = \sum_{i=0}^k c_i
    \epsilon^{-s_i} \int_0^\infty t^{-s_i}f(t) dt + o(t^{-s_k}).
  \]
\end{prop}
Thus, if we simply choose $f$ so that its moments $\int_0^\infty t^{-s_i} f(t)
dt $ vanish for $i \neq k$ and normalize it so that $\int_0^\infty t^{-s_k} f(t)
dt = 1$, we gain access to the coefficient $c_k$.

By asymptotic decay of such a function $f$ towards $\infty$, it has absolutely
convergent Laplace transform $F$. As $\left|  f(t) \tr e^{- \epsilon t \Delta}  \right|$ is absolutely integrable as well, we can apply the Fubini-Tonelli theorem to obtain
\begin{align*}
    \int_0^\infty \tr  e^{-\epsilon t \Delta} f(t) dt &= \sum_{\lambda \in \sigma(\Delta)} \int_0^\infty e^{-\lambda \epsilon t} f(t) dt \\
    &= \tr  F(\Delta \epsilon),
\end{align*}
where $\sigma(\Delta)$ is the spectrum (with multiplicities) of $\Delta$. Therefore, 
Proposition \ref{prop:convolutionasymptotics} leads us to the conclusion that
\begin{equation} \label{eq:fulltrace}
    \tr  F(\epsilon \Delta) = c_k \epsilon^{- s_k}  + O(\epsilon^{-s_{k + 1}}).
\end{equation}

The following proposition  shows that we retain the same asymptotics if we replace the trace of $F(\epsilon \Delta)$ by a sum over a finite part of the spectrum, provided we scale $\epsilon$ accordingly.

\begin{prop} \label{prop:partialtrace}
    If $f$ and $\Delta$ are as in Proposition \ref{prop:convolutionasymptotics} and
    additionally
    \[
      \int_0^\infty t^{-s_i} f(t) dt = 0 \quad \text{(for all $i<k$)}, 
    \]
    then the Laplace transform $F$ satisfies
    \[
      \sum_{\lambda > \Lambda} F( \lambda \log \Lambda^m / \Lambda) = o((\log
      \Lambda^m / \Lambda)^{-s_k})
    \]
    for all $m > s_0 - s_k$
\end{prop}
\begin{proof}
  By the fact that $\supp f \subset [1,\infty)]$, its Laplace transform decays
  as $F(s) = O(s^{-1} e^{-s})$ towards $\infty$. As $\sum_{\lambda > \Lambda}
  e^{-\epsilon(\lambda-\Lambda) t}$ is monotone decreasing in $t$, we see that
  \[
\sum_{\lambda > \Lambda} F(\epsilon \lambda) = O \br{\sum_{\lambda > \Lambda}
(\epsilon \Lambda)^{-1} e^{- \epsilon \lambda}} = O \br{(\epsilon \Lambda)^{-1} e^{-
\epsilon \Lambda} \sum_{\lambda > \Lambda} e^{-\epsilon(\lambda - \Lambda) t}}
\]
for any $t \leq 1$, as $\Lambda \epsilon \to \infty$. This, in turn, is $O(
(\epsilon \Lambda)^{-1} e^{-\epsilon \Lambda(1-t)} \epsilon^{- s_0})$. If, then,
$\epsilon(\Lambda) = m \Lambda^{-1} \log \Lambda$ for any $m > (s_0 - s_k)$, the remainder $\sum_{\lambda > \Lambda}
  F(\epsilon \lambda)$ will be $o(\epsilon^{-s_k})$.
\end{proof}

This completes the proof of Theorem \ref{thm:maintheorem} on page \pageref{pf:maintheorem}.

\subsection{Explicit formulas for the first two poles}
One reason to look for series that converge to zeta residues is to obtain geometric information from finite-dimensional approximations of the Laplacian on a Riemannian manifold. For instance, if $\Delta$ is the Laplacian, the first two residues are proportional to the volume and the total scalar curvature, respectively. 

The first pole is a classical object of study. Its residue is expressed by Dixmier's singular trace, and is used, for instance, for the zeta regularization of divergent series. It is connected to counting asymptotics by the Wiener-Ikehara theorem, and for a Laplace operator on a compact Riemannian manifold it is proportional to the volume.

Our Theorem \ref{thm:maintheorem} provides the following formula for the first
residue.

\begin{corr} \label{corr:firstpole}
    If $\Delta$ is a spectrally elliptic operator with heat spectrum bounded
    from above by $s_0 \in \R$, then if $\epsilon(\Lambda) = \log \Lambda / \Lambda$,
    \[
	\Gamma(s_0) \res_{s=1} \zeta(\Delta^{s_0},s) =  \lim_{\Lambda \to \infty}
  \frac{ \epsilon(\Lambda)^{s_0} }{e \Gamma(1-s_0,1)}
  \tr_\Lambda \frac{ e^{-\Delta \epsilon(\Lambda) } }{1 +
    \Delta \epsilon(\Lambda)}.
\]
\end{corr}
\begin{proof}
  Use $f = [t \geq 1] e^{-t} $, with Laplace transform $F(s) = e^{-1-s}/(1+s)$,
  in Proposition \ref{prop:convolutionasymptotics}. Then, divide by $\int_0^\infty  t^{-s_0} f(t) dt$ to satisfy the conditions of
  Proposition \ref{prop:partialtrace}.
\end{proof}
\begin{remark}
  The present series converges faster in general than the logarithmic trace suggested by Weyl's asymptotic formula; the remainder is $O(\epsilon^{s_0 -  s_1}) = O(\Lambda^{s_1-s_0} (\log \Lambda)^{s_0 - s_{1}})$, whereas for e.g. the Laplacian on the circle the logarithmic trace of $\Delta^{-1/2}$ convergences only as $\sum_{n=1}^N \frac1{n \log N} - 1 = \frac{\gamma + O(N^{-1})}{\log N}$.
\end{remark}

Now, the second pole is of particular interest because it is the first pole for which no asymptotic residue formula was previously known and because it provides a way to calculate the total scalar curvature of a Riemannian manifold from partial spectra of the Laplacian. 
\begin{corr} \label{corr:secondpole}
  If $\Delta$ is a spectrally elliptic operator with heat spectrum $\{ s_0,s_0
  - 1,\dotsc \}$, then if $\epsilon(\Lambda) = 2 \log \Lambda / \Lambda$,
  \[
    \res_{s=s_1} \zeta(\Delta,s) \Gamma(s)= \lim_{\Lambda \to \infty}
    \frac{\epsilon(\Lambda)^{s_0-1}}{e \Gamma(2-s_0,1)} \tr_\Lambda \br{
      \frac{2^{s_0} e^{-2 \Delta \epsilon(\Lambda)}}{1 + 2 \Delta \epsilon(\Lambda)} - \frac{e^{-
          \Delta \epsilon(\Lambda)}}{1 + \Delta \epsilon(\Lambda)}}.
  \]
\end{corr}
\begin{proof}
  Use $f = [t \geq 1] e^{-t} - 2^{s_0 - 1} [t \geq 2] e^{-t/2}$, which clearly satisfies the conditions of Proposition  \ref{prop:convolutionasymptotics} and has vanishing moment $\int_0^\infty t^{-s_0} f(t) dt$. Therefore, if we divide by $\int_0^\infty t^{1-s_0} f(t) dt = - \Gamma(2-s_0,1)$, the result will be as in Proposition \ref{prop:partialtrace}. The Laplace transform of $f$ is $F(s) = e^{-1-s}/(1+s) - 2^{s_0} e^{-1-2s}/(1+2s)$.
\end{proof}

\section{Localization}
If $\Delta$ is spectrally elliptic, Theorem \ref{thm:maintheorem} allows us to calculate its zeta residues as a series in its eigenvalues. However, the classical theory of elliptic pseudodifferential operators assigns to them not just the total zeta residues, but rather a set of zeta \emph{densities}. To be precise, the somewhat simpler situation for elliptic differential operators is as follows.

Let $M$ be a Riemannian manifold of dimension $n$ equipped with a smooth Hermitian vector bundle $E$, and let $\Delta$ be a positive, self-adjoint, elliptic differential operator of order $m>0$, acting inside the Hilbert space of square-integrable sections of $E$.
Then, the following localized asymptotics are available.
\begin{enumerate}
\item The operator $e^{-t \Delta}$ is given by a smooth kernel $k(x,y,t): E_y \to E_x$, and as $t \to 0^+$ its restriction to the diagonal admits an asymptotic expansion in smooth sections $k_j$ of $\End E$,
  \[
    k(x,x,t) \sim \sum_{j=0}^{\infty} k_j(x) t^{(j-n)/m}.
  \]
\item For any continuous section $h$ of $\End E$ and any $j \geq 0$, the following residues exist and satisfy
  \[
    \res_{s=(j-n)/m} \Gamma(s) \tr h \Delta^{-s} = \int_M \tr (h(x) k_j(x)) d \vol_M(x).
  \]
\end{enumerate}
We say that $h$ \emph{localizes} the residue trace, and we are interested in expressing the localized residues in a fashion similar to Theorem \ref{thm:maintheorem}. 
We will solve this localization problem in a slightly more general setting, along the lines of the previous treatment of the zeta function residues.

Let $\Delta$ be a spectrally elliptic operator on a Hilbert space $\mathcal{H}$ and let $h \in \mathcal{B}(\mathcal{H})$ be bounded. Then, there exist $s_0>\dotsb>s_{k+1}$ such that, as $t \to 0^+$,
\[
  \tr  h e^{-t \Delta} - \sum_{i=0}^k t^{-s_k} c_i(h) = O(t^{-s_{k+1}}).
\]

As before, let $\{s_i\}_{i=1}^k$ be a set of decreasingly ordered reals and let $F$ be the Laplace transform of a piecewise continuous function $f$ supported in $[1,\infty)$, which is
$O(t^{-m})$ for all $ m \in \R$ towards $ \infty $, satisfies $\int_0^\infty
t^{-s_i} f(t) dt = 0$ for all $i < k$, and is normalized to satisfy $\int_0^\infty t^{-s_k} f(t) dt = 1$.

\begin{thm} \label{thm:localization}
  Let $\Delta \in \mathcal{D}(\{s_i\}_{i=1}^k)$ be a spectrally elliptic operator on a Hilbert space $\mathcal{H}$ and let $h \in \mathcal{B}(\mathcal{H})$ be bounded.
  Then, for any orthonormal basis $\{ \phi_\lambda \}_{\lambda \in
    \sigma(\Delta)}$ of $\mathcal{H}$ diagonalizing $\Delta$, we have
  \[
    \res_{s=s_k} \Gamma(s) \tr h \Delta^{-s} = \lim_{\Lambda \to \infty}
    \epsilon(\Lambda)^{s_k} \sum_{\mathmakebox[\stdsummationwidth]{\lambda \in \sigma(\Delta) \cap [0,\Lambda]}} \langle h
    \phi_\lambda, \phi_\lambda \rangle F( \lambda \epsilon(\Lambda)),
  \]
  where $\epsilon(\Lambda) \eqdef \Lambda^{-1} \log \Lambda^m $ for any $m > s_k - s_0$.
\end{thm}

\begin{proof}
  As before, Proposition \ref{prop:convolutionasymptotics} holds for any such function $f$. That is,
  \[
    \int_0^\infty \tr  h e^{- \epsilon t \Delta} f(t) dt = c_k(h) \epsilon^{- s_k}  + \| h \|_{\mathrm{op}} O(\epsilon^{- s_{k+1} }).
  \]
  Moreover, $\sum_{\lambda \in \sigma(\Delta)} |f(t) e^{- t \lambda}  \langle h \phi_\lambda, \phi_\lambda \rangle  | \leq |f(t)| \| h \|_{\mathrm{op}} \tr  e^{- t \Delta}$ is in $L^1(0,\infty)$. Therefore, by dominated convergence,
  \[
    \int_0^\infty \tr  h e^{- \epsilon t \Delta} f(t) dt = \sum_{\lambda \in \sigma(\Delta)} \langle h \phi_\lambda, \phi_\lambda \rangle F(\epsilon \lambda)
  \] 
  and by Proposition \ref{prop:partialtrace}, the rest term decays like 
  \[
    \sum_{\mathmakebox[\stdsummationwidth]{\lambda \in \sigma(\Delta) \cap [\Lambda,\infty)}}
    F(\lambda \epsilon(\Lambda)) \langle h \phi_\lambda, \phi_\lambda \rangle
    = \| h \|_{\mathrm{op}} O( (\epsilon(\Lambda))^{-s_{k+1}} ).
  \]
  We have $\res_{s=s_k} \Gamma(s) \tr h \Delta^{-s} = c_k(h)$ and the conclusion follows.
\end{proof}

\section{Final remarks and suggestions}
\begin{remark}
  The method of Theorem \ref{thm:localization} may be applied to obtain a reasonable definition of scalar curvature in \emph{finite-dimensional} noncommutative geometry. Let $(A,H,D)$ be a spectral triple such that $D^2$ is spectrally elliptic with heat spectrum $\{(j-n)/2 \}_{j \in \Z_{>0}}$. In \cite{Connes2008a} the scalar curvature functional of such a spectral triple was defined to be the map 
  \[
    a \mapsto R(a) \eqdef \res_{s=n-2} \Gamma(s) \tr a D^{-s}.
  \]
  If $H$ is finite-dimensional but a dimension spectrum of $D$ is specified in advance (e.g. by modelling considerations), this residue always vanishes and Theorem \ref{thm:localization} suggests it should perhaps be replaced by
  \[
    a \mapsto R_\Lambda(a) \eqdef \sum_{\lambda \in \sigma(D^2)} \langle a
    \phi_\lambda, \phi_\lambda \rangle \frac{  F(\lambda \log \Lambda^m / \Lambda)
    }{ (\log \Lambda^m / \Lambda)^{n/2-1}},
  \]
  where $\Lambda \eqdef \| D^2 \|^\rho$ for any $0 < \rho < 1$, $\{\phi_\lambda
  \mid D^2 \phi_\lambda = \lambda \phi_\lambda \}$ is an orthonormal basis of $H$
  and $F$ is as in Corollary \ref{corr:secondpole}.
\end{remark}

\begin{remark}
  The improved convergence in Corollary \ref{corr:firstpole}, combined with Connes' Hochschild class formula \cite[Thm 4.2.$\gamma$.8]{Connes} for the Chern character of the Fredholm module associated to a spectral triple $(A,H,D)$, allows one to estimate some $\operatorname{KK}$-theoretic index pairings numerically. If the square of the operator $D$ is spectrally elliptic, the index pairings can be expressed as a series in the spectrum of $D^2$ and the coefficients $\tr \pi_{\lambda} a_0 [D,a_1] \dotsb [D,a_n] \pi_{\lambda}$, where $\pi_\lambda$ projects onto the eigenspace associated with the eigenvalue $\lambda$ and $a_0,\dotsc,a_n \in A$. A similar statement holds in the presence of a grading.
\end{remark}

A previous version of this paper was circulated on the arXiv under the same
title. This version has been
highly simplified and the present approximations converge exceedingly faster,
both asymptotically and in all tested examples.

\bibliographystyle{alpha}
\bibliography{library}

\end{document}